\documentclass{monsky2009}
\usepackage{amssymb}

\usepackage{bm}

\newenvironment{conjecture*}[1][]{\textbf{Conjecture #1\hspace{.3em}}}{}
\newenvironment{theorem*}[1]{\textbf{#1}\itshape \hspace{.3em}}{\upshape}
\newenvironment{remark*}[1]{\textbf{#1}\itshape \hspace{.3em}}{\upshape}
\newenvironment{example*}[1]{\textbf{#1}\itshape \hspace{.3em}}{\upshape}
\newenvironment{proof}[1][]{\textbf{Proof #1\hspace{.3em}}}{}

\newtheorem{definition}{Definition}[section]
\newtheorem{theorem}[definition]{Theorem}
\newtheorem{lemma}[definition]{Lemma}

\newtheorem{corollary}[definition]{Corollary}

\newcounter{kpremark}
\newcounter{proofitem}
\newcommand{\mod}[1]{\ensuremath{\hspace{.5em}(#1)}}

\newcommand{\modd}{\ensuremath{M(\mathit{odd})}}

\begin{document}

\begin{frontmatter}



%\title{}


%\author{}

\title{A characteristic 2 recurrence related to \bm{$U_{5}$}, with a Hecke algebra application}
\author{Paul Monsky}

\address{Brandeis University, Waltham MA  02454-9110, USA. monsky@brandeis.edu}

\begin{abstract}
In arXiv:1603.03910 [math.NT] we introduced some $C_{n}$ in $Z/2[t]$ defined by a linear recurrence and showed that each $C_{n}$, $n\equiv 0 \mod{4}$, is a sum of $C_{k}$, $k<n$. Combining this with results from arXiv:1508.07523 [math.NT] we proved that the space $K$, consisting of those odd mod~2 modular forms of level $\Gamma_{0}(3)$ that are annihilated by the operator $U_{3}+I$, has a basis $m_{i,j}$ ``adapted to $T_{7}$ and $T_{13}$'' in the sense of Nicolas and Serre.  (And so the ``completed shallow Hecke algebra'' attached to $K$ is a power series ring in $T_{7}$ and $T_{13}$.)

This note derives analogous results in level $\Gamma_{0}(5)$. Now $U_{3}+I$ is replaced by $U_{5}+I$, and the operators $T_{7}$ and $T_{13}$ by $T_{3}$ and $T_{7}$. In place of level $\Gamma_{0}(3)$ results from 1508.07523, we use level $\Gamma_{0}(5)$ results from arXiv:1603.07085 [math.NT]. A linear recurrence again plays the key role. Now $C_{n+6} = C_{n+5} + (t^{6}+t^{5}+t^{2}+t)C_{n}+t^{n}(t^{2}+t)$, $C_{0}=0$, $C_{1}=C_{2}=1$, $C_{3}=t$, $C_{4}=t^{2}$, $C_{5}=t^{4}+t^{2}+t$, and we prove that each $C_{n}$, $n\equiv 0$ or $2\mod{6}$ is a sum of $C_{k}$, $k<n$.
\end{abstract}

%\begin{keyword}

%\PACS 
%\end{keyword}
\end{frontmatter}

% main text

\section{Introduction. The polynomials \bm{$C_{n}$}}
\label{section1}

The first 4 sections of this note, apparently unconnected with modular forms, consist of calculations in a polynomial ring $Z/2[r]$ that might seem unmotivated. The motivation in fact comes from \cite{4}.  In that note I introduce commuting Hecke operators $T_{p}:Z/2[[x]]\rightarrow Z/2[[x]]$, $p$ an odd prime, together with a subspace, $\modd$, of $Z/2[[x]]$ stabilized by the $T_{p}$, $p\ne 5$. $\modd$ is the ``space of odd mod~2 modular forms of level $\Gamma_{0}(5)$,'' and contains the reductions $F$ and $G$ of the expansions at infinity of $\Delta(z)$ and $\Delta(5z)$. I construct subspaces $N2\supset N1$ of $\modd$, also stabilized by these $T_{p}$, and write $N2/N1$ as a direct sum $N2a\oplus N2b$. I then show that $N2a$ and $N2b$ are stabilized by the $T_{p}$, $p\equiv 1,3,7,9\mod{20}$ and that $N2a$ has a basis $m_{i,j}$ ``adapted to $T_{3}$ and $T_{7}$'' in the sense of Nicolas and Serre, with $m_{0,0}=F$.

Now let $U_{5}:Z/2[[x]]\rightarrow Z/2[[x]]$ be the map $\sum c_{n}x^{n}\rightarrow \sum c_{5n}x^{n}$. It's easy to see that $U_{5}$ stabilizes $\modd$ and commutes with the $T_{p}$, $p\ne 5$. So the kernel, $K$, of $U_{5}+I$ acting on $\modd$ is stabilized by \emph{all} $T_{p}$, $p\ne 5$ (or $2$), and there is a shallow Hecke algebra attached to $K$, generated by these $T_{p}$.

The goal of this note is to show that $K$, like $N2a$, has a basis $m_{i,j}$ adapted to $T_{3}$ and $T_{7}$; now $m_{0,0}=F+G$. In fact we show that $K\subset N2$, and that the projection map from $N2$ to $N2a$ maps $K$ bijectively to $N2a$. Since the projection map preserves the action of $T_{3}$ and $T_{7}$, $K$ has the desired basis. A consequence is that there is a faithful action of $Z/2[[X,Y]]$ on $K$ with $X$ and $Y$ acting by $T_{3}$ and $T_{7}$, and that each $T_{p}:K\rightarrow K$, $p\ne 2$ or $5$ is multiplication by an element of $(X,Y)$. In a more elaborate language this says that the ``completed shallow Hecke algebra'' attached to $K$ is a power series ring in $T_{3}$ and $T_{7}$.

Now the statement that $K\subset N2$ and projects bijectively from $N2/N1$ to $N2a$ can be phrased without invoking modular forms; it is just algebra in a polynomial ring. Namely there is an $r$ in $Z/2[[x]]$ with $r^{2}+r=F+G$, and using results from \cite{4} one can show that the principal actors, $F$, $G$, $\modd$ $U_{5}$, $K$, $N1$, $N2a$, $N2b$ can be defined purely in terms of this $r$. In sections \ref{section1}--\ref{section4} we take $r$ to be an indeterminate over $Z/2$, and define these principal actors. (But $U_{5}$ will be called $U$.) We then show that $K\subset N2$ and that $K\rightarrow N2a$ is bijective. In section \ref{section5}, the actors drop their masks, revealing their connections with modular forms, and we conclude that the kernel of $U_{5}+I$ acting on $\modd$ has the desired basis.

Here's a preview of the section \ref{section1}--\ref{section4} calculations in $Z/2[r]$. $F$ and $G$ are $r(r+1)^{5}$ and $r^{5}(r+1)$. $\modd$ is spanned by the  $(r^{2}+r)r^{2n}$, and $U:Z/2[r]\rightarrow Z/2[r]$ is a certain operator with $U(Gf)=FU(f)$. In section \ref{section1} we show that $U+I$ takes $(r^{2}+r)r^{2n}$ to $(r^{2}+r)C_{n}(r^{2})$ where the $C_{n}$ lie in $Z/2[t]$ and satisfy a certain recurrence. In section \ref{section2} we use this recurrence to show that each $C_{n}$ with $n\equiv 0$ or $2\mod{6}$ is a sum of $C_{k}$, $k<n$, and consequently that there are $g_{n}$ of degree $n$, one for each $n\equiv 0$ or $2\mod{6}$, with $f_{n}=(r^{2}+r)g_{n}(r^{2})$ forming a basis of $K$. In Theorem \ref{theorem3.5} of section \ref{section3} we show that these $g_{n}$ may be chosen to satisfy somewhat more stringent conditions. In section \ref{section4} we re-examine the arguments of section \ref{section2}. We construct $Z/2[G^{2}]$ submodules $N1$ and $N2$ of $\modd$ with bases $\{G\}$ and $\{G,F,F^{2}G,F^{3},F^{4}G\}$ and show that $K\subset N2$. We then define certain $J_{k}$, $(k,10)=1$, in $N2$ with $J_{k+10}=G^{2}J_{k}$. $N2a$ and $N2b$ are the subspaces of $N2/N1$ spanned by the $J_{k}$ with $k\equiv 1,3,7,9\mod{20}$ and the $J_{k}$ with $k\equiv 11,13,17,19\mod{20}$. We use Theorem \ref{theorem3.5} to get information about the image of $f_{n}$ under the composite map $N2\rightarrow N2/N1\rightarrow N2a$, and establish the desired bijection.

\begin{definition}
\label{def1.1}
$F$ and $G$ in $Z/2[r]$ are $r(r+1)^{5}$ and $r^{5}(r+1)$. Note that $F+G=r(r+1)$ and $(F+G)^{6}+FG=0$.
\end{definition}

\begin{definition}
\label{def1.2}
$\sigma : Z/2[r]\rightarrow Z/2[r]$ is semi-linear if it is $Z/2$-linear, and $\sigma(Gf)=F\sigma (f)$.
\end{definition}

Since a basis of $Z/2[r]$ as module over $Z/2[G]=Z/2[r^{6}+r^{5}]$ is $1,r,r^{2},r^{3},r^{4},r^{5}$, a semi-linear map is determined by the images of these 6 elements. And for any choice of 6 elements to be images there is a corresponding semi-linear map.

\begin{definition}
\label{def1.3}
$U : Z/2[r]\rightarrow Z/2[r]$ is the semi-linear map taking $1,r,r^{2},r^{3},r^{4}$ and $r^{5}$ to $1,r,r^{2},r^{3}+r^{2}+r,r^{4}$ and $r^{5}+r^{4}+r$. Since $U(1)=1$, $U(G^{n})=F^{n}$.
\end{definition}

\begin{lemma}
\label{lemma1.4}
$U(f^{2})=(U(f))^{2}$
\end{lemma}

\begin{proof}
Since $1,r,r^{2},r^{3},r^{4},r^{5}$ form a basis of $Z/2[r]$ as $Z/2[G]$-module, and $U(G)=F$, it suffices to prove this when $f$ is $1,r,r^{2},r^{3},r^{4}$ or $r^{5}$. The result is immediate for $1$, $r$ and $r^{2}$. And:
\vspace{-2ex}
\begin{eqnarray*}
r^{6} &=& G+r^{5}.\ \mbox{So }U(r^{6})=F+(r^{5}+r^{4}+r) = r^{6}+r^{4}+r^{2} = (U(r^{3})^{2}.\\
r^{8} &=& (r^{2}+r)G+r^{6}.\ \mbox{So }U(r^{8})=(r^{2}+r)F+r^{6}+r^{4}+r^{2} = r^{8} = (U(r^{4})^{2}.\\
r^{10} &=& (r^{4}+r^{3})G+r^{8}.\ \mbox{So }U(r^{10})=(r^{4}+r^{3}+r^{2}+r)F+r^{8} \\
&=& r^{2}(r+1)^{8}+r^{8} = r^{10}+r^{8}+r^{2} =(U(r^{5})^{2}.
\end{eqnarray*}
\qed
\end{proof}

\begin{lemma}
\label{lemma1.5}\hspace{2em}\\
\vspace{-2ex}
\begin{enumerate}
\item[(a)] $U(r^{n+6})=U(r^{n+5})+(r^{6}+r^{5}+r^{2}+r)U(r^{n})$.
\item[(b)] $U((r^{2}+r)r^{2n})=(r^{2}+r)A_{n}(r^{2})$ for some $A_{n}$ in $Z/2[t]$.
\end{enumerate}
\end{lemma}

\begin{proof}
$r^{6}=r^{5}+G$. Multiplying by $r^{n}$, applying $U$ and using semi-linearity, we get (a).  Now $U(r^{2}+r)=r^{2}+r$, while $U(r^{4}+r^{3})=(r^{2}+r)(r^{2}+1)$. So for $n=0$ and $n=1$ we have (b) with $A_{0}=1$ and $A_{1}=t+1$. Suppose $n\ge 2$. By (a), $U((r^{2}+r)r^{2n})= (r^{6}+r^{5}+r^{2}+r)(U(r^{2n-4})$, which by Lemma \ref{lemma1.4} is $(r^{2}+r)(r^{4}+1)(U(r^{n-2}))^{2}$. So if we write $U(r^{n-2})$ as $g(r)$, we get (b) with $A_{n}=(t^{2}+1)g(t)$.
\qed
\end{proof}

\begin{lemma}
\label{lemma1.6}
Let $A_{n}$ be as in Lemma \ref{lemma1.5}. Then:
\vspace{-2ex}
\begin{enumerate}
\item[(a)] $A_{n+6}=A_{n+5}+(t^{6}+t^{5}+t^{2}+t)A_{n}$.
\item[(b)] $A_{0},A_{1},A_{2},A_{3},A_{4},A_{5}$ are $1,t+1,t^{2}+1,t^{3}+t,t^{4}+t^{2},t^{5}+t^{4}+t^{2}+t$.
\end{enumerate}
\end{lemma}

\begin{proof}
$r^{12}=r^{10}+G^{2}$. Multiplying by $(r^{2}+r)r^{2n}$, applying $U$, and then dividing by $r^{2}+r$ we find that $A_{n+6}(r^{2})=A_{n+5}(r^{2})+(r^{12}+r^{10}+r^{4}+r^{2})A_{n}(r^{2})$, giving (a). $A_{0}$ and $A_{1}$ are $1$ and $t+1$. When $n=2,3,4,5$, then $U(r^{n-2})$ is $1,r,r^{2},r^{3}+r^{2}+r$, and we use the final sentence in the proof of Lemma \ref{lemma1.5} to get (b). 
\qed
\end{proof}

\begin{definition}
\label{def1.7}
$C_{n}$ in $Z/2[t]$ is $A_{n}+t^{n}$, with $A_{n}$ as in Lemma \ref{lemma1.5}.
\end{definition}

\begin{theorem}
\label{theorem1.8}\hspace{2em}\\
\vspace{-2ex}
\begin{enumerate}
\item[(a)] $C_{n+6}=C_{n+5}+(t^{6}+t^{5}+t^{2}+t)C_{n}+t^{n}(t^{2}+t)$.
\item[(b)] $C_{0},C_{1},C_{2},C_{3},C_{4},C_{5}$ are $0,1,1,t,t^{2},t^{4}+t^{2}+t$.
\item[(c)] $U+I$ takes $(r^{2}+r)r^{2n}$ to $(r^{2}+r)C_{n}(r^{2})$.
\end{enumerate}
\end{theorem}

\begin{proof}
(a) and (b) follow from (a) and (b) of Lemma \ref{lemma1.6}, while (c) comes from Lemma \ref{lemma1.5}(b).
\qed
\end{proof}

\begin{lemma}
\label{lemma1.9}\hspace{2em}\\
\vspace{-2ex}
\begin{enumerate}
\item[(a)] If $n\equiv 1$ or $5\mod{6}$, degree $C_{n}=n-1$.
\item[(b)] If $n\equiv 3$ or $4\mod{6}$, degree $C_{n}=n-2$.
\item[(c)] If $n\equiv 0$ or $2\mod{6}$, degree $C_{n}\le n-2$.
\end{enumerate}
\end{lemma}

\begin{proof}
For $n\le 5$ we use Theorem \ref{theorem1.8}(b).  To show, by induction, that the results hold for $n+6$, look at (a) of Theorem \ref{theorem1.8}.  The induction assumption tells us that the first term on the right has degree $\le n+4$, and the same is true for the last term. When $n\equiv 1$ or $5\mod{6}$, the middle term, by the induction assumption, has degree $6+(n-1)=n+5$, so $C_{n+6}$ has degree $n+5$. When $n\equiv 0$ or $2\mod{6}$, the middle term has degree $\le 6+(n-2)=n+4$, so degree $C_{n+6}\le n+4$.  Finally when $n\equiv 3$ or $4\mod{6}$, the first and last terms on the right have degree $\le n+3$ while the middle term has degree $6+(n-2)=n+4$. So $C_{n+6}$ has degree $n+4$.
\qed
\end{proof}

\begin{lemma}
\label{lemma1.10}
If $C_{n}$ is a $Z/2$-linear combination of $C_{k}$, $k<n$, then $n\equiv 0$ or $2\mod{6}$.
\end{lemma}

\begin{proof}
If $n\equiv 1$ or $5\mod{6}$, then $C_{n}$ has degree $n-1$, while each $C_{k}$, $k<n$, has degree $<n-1$. Similarly when $n\equiv 3$ or $4\mod{6}$, $C_{n}$ has degree $n-2$, while each $C_{k}$, $k<n$, has degree $<n-2$.
\qed
\end{proof}

\section{A key property of the \bm{$C_{n}$}}
\label{section2}

We'll prove a converse, Theorem \ref{theorem2.11}, to Lemma \ref{lemma1.10}; if $n\equiv 0$ or $2\mod{6}$, then $C_{n}$ is a $Z/2$-linear combination of $C_{k}$, $k<n$. (This is one of several related conjectures about recurrences found in \cite{1}. Peter M\"{u}ller, in a short computer calculation, showed it to be true for $n<10,000$.)

\begin{lemma}
\label{lemma2.1}
$U$ takes $F+G$, $(F+G)^{3}$ and $(F+G)^{5}$ to $F+G$, $(F+G)^{3}$ and $(F+G)^{5}+F$.
\end{lemma}

\begin{proof}
Since $F+G=r^{2}+r$, the first result holds. Also, $U(F+G)^{3}=U(r^{6}+r^{5}+r^{4}+r^{3})=U(G+r^{4}+r^{3})= F+r^{4}+r^{3}+r^{2}+r=r^{6}+r^{5}+r^{4}+r^{3}$. Finally, $(F+G)^{5}=(r^{2}+r)^{5}=(r^{4}+1)G$; $U$ takes this to $r^{4}F+F=(r^{2}+r)^{5}+F$.
\qed
\end{proof}

\begin{lemma}
\label{lemma2.2}
$U$ takes $1,F,F^{2},F^{3},F^{4},F^{5}$ to $1,G,G^{2},G^{3},G^{4},G^{5}+F$.
\end{lemma}

\begin{proof}
Since $U$ fixes $F+G$, and $U(G)=F$, $U(F)$ must be $G$. By Lemma \ref{lemma1.4}, $U(F^{2})=G^{2}$, $U(F^{4})=G^{4}$.  Since $U$ is semi-linear it interchanges $F^{2}G$ and $FG^{2}$, as well as $F^{4}G$ and $FG^{4}$. Now $F^{3}+G^{3}=(F+G)^{3}+F^{2}G+FG^{2}$. Lemma \ref{lemma2.1} then shows that $U$ fixes $F^{3}+G^{3}$, and since $U(G^{3})=F^{3}$, $U(F^{3})=G^{3}$. Similarly, since $F^{5}+G^{5}=(F+G)^{5}+F^{4}G+FG^{4}$, Lemma \ref{lemma2.1} shows that $U$ takes $F^{5}+G^{5}$ to $F^{5}+G^{5}+F$, and since $U(G^{5})=F^{5}$, $U(F^{5})=G^{5}+F$.
\qed
\end{proof}

\begin{lemma}
\label{lemma2.3}
Let $\alpha$ be the $Z/2$-linear isomorphism $Z/2[F]\rightarrow Z/2[G]$ with $\alpha(F^{n})=G^{n}$. Then if $P_{n}$ is either $\alpha(F^{n})=G^{n}$ or $U(F^{n})$, 
\[
(\star)\qquad
P_{n+6}+F^{2}P_{n+4}+F^{4}P_{n+2}+F^{6}P_{n}+FP_{n+1}=0.
\]
\end{lemma}

\begin{proof}
As we saw in Definition \ref{def1.1}, $(F+G)^{6}+FG=0$. Multiplying by $G^{n}$ and expanding we get $(\star)$ with $P_{n}=G^{n}=\alpha(F^{n})$. Multiplying instead by $F^{n}$, and applying the semi-linear operator $U$, we get $(\star)$ with $P_{n}=U(F^{n})$.
\qed
\end{proof}

\begin{definition}
\label{def2.4}
For $f$ in $Z/2[F]$, $T(f)=U(f)+\alpha(f)$.
\end{definition}

\begin{lemma}
\label{lemma2.5}\hspace{2em}\\
\vspace{-2ex}
\begin{enumerate}
\item[(a)] $T$ takes $1,F,F^{2},F^{3},F^{4}$ and $F^{5}$ to $0,0,0,0,0$ and $F$.
\item[(b)] $T$ stabilizes $Z/2[F]$. In fact, $T(F^{n})$ is a sum of $F^{k}$ with each $k\equiv n\mod{2}$, and $\le n-4$.
\end{enumerate}
\end{lemma}

\begin{proof}
(a) is immediate from Lemma \ref{lemma2.2}, and in particular the second assertion in (b) holds for $n\le 5$. Now let $P_{n}=T(F^{n})$. By Lemma \ref{lemma2.3} the $P_{n}$ satisfy the recursion $(\star)$ above. An induction on $n$ completes the proof of (b).
\qed
\end{proof}

\begin{lemma}
\label{lemma2.6}
Let $u_{0},u_{1},u_{2},u_{4},u_{5}$ be $F+G,(F+G)^{3}+G,G,(F+G)^{2}G,(F+G)^{4}G+(F+G)FG$. Then the $u_{i}$ are linearly independent over $Z/2[G]$ and $u_{i}=(r^{2}+r)g(r^{2})$ for some $g$ of degree $i$.
\end{lemma}

\begin{proof}
Since $1,F,F^{2},F^{3}$ and $F^{4}$ are linearly independent over $Z/2[G]$, the first assertion holds. The $g$ corresponding to $u_{0}$ and $u_{2}$ are evidently $1$ and $t^{2}$. Since $(F+G)^{2}=r^{4}+r^{2}$, the $g$ corresponding to $u_{1}$ is $(t^{2}+t)+t^{2}=t$. Since $(F+G)G=r^{8}+r^{6}$, the $g$ corresponding to $u_{4}$ is $t^{4}+t^{3}$. Since $G(F+G)^{3}+FG=(r^{6}+r^{5})(r^{4}+r^{3}+r^{2}+r)=r^{10}+r^{6}$, the $g$ corresponding to $u_{5}$ is $t^{5}+t^{3}$.
\qed
\end{proof}

We now fix $m\ge 0$.

\begin{definition}
\label{def2.7}
$L$ is the space spanned by the $u_{i}G^{2n}$ with $i\in \{0,1,2,4,5\}$ and $0\le n\le m$. $L^{*}$ consists of the $(r^{2}+r)g(r^{2})$, where $g$ in $Z/2[t]$ has degree $\le 6m+5$.
\end{definition}

\begin{lemma}
\label{lemma2.8}
$L$ has dimension $5m+5$, and $L\subset L^{*}$.
\end{lemma}

\begin{proof}
The linear independence of the $u_{i}$ over $Z/2[G]$ gives the first result. Since $G^{2n}=(r^{12}+r^{10})^{n}$, and $n\le m$, the last part of Lemma \ref{lemma2.6} shows that $L\subset L^{*}$.
\qed
\end{proof}

\begin{remark*}{Remark}
Let $\modd$ consist of all elements of $Z/2[r]$ of the form $(r^{2}+r)g(r^{2})$, $g$ in $Z/2[t]$. Lemma \ref{lemma1.5}(b) shows that $U$ stabilizes $\modd$. Also, if the $A_{n}$ are as in Lemma \ref{lemma1.5}, then Lemma \ref{lemma1.6} and an induction show that the degree of $A_{n}$ is $\le n$, and it follows that $U$ stabilizes $L^{*}$.  (But when $m\ge 2$, $U$ does \emph{not} stabilize $L$. For $G^{5}=G^{4}u_{2}$ is in $L$, but $U(G^{5})=F^{5}$ is not even a $Z/2[G]$-linear combination of $u_{0},u_{1},u_{2},u_{4}$ and $u_{5}$.)
\end{remark*}

\begin{lemma}
\label{lemma2.9}
For $0\le i\le 4$, $(U+I)^{2}=U^{2}+I$ maps $F^{i}G^{k}$ to $F^{i}T(F^{k})$.
\end{lemma}

\begin{proof}
$U(F^{i}G^{k})=F^{k}U(F^{i})$; since $i\le 4$ this is $F^{k}G^{i}$. Then $U^{2}(F^{i}G^{k})=F^{i}U(F^{k})=F^{i}(G^{k}+T(F^{i}))$, and the result follows.
\qed
\end{proof}

\begin{theorem}
\label{theorem2.10}
Let $K_{m}$ be the kernel of $U+I:L^{*}\rightarrow L^{*}$. (The remark above shows that $U+I$ stabilizes $L^{*}$.) Then the dimension of $K_{m}$ is $\ge 2m+2$.
\end{theorem}

\begin{proof}
Each $u_{i}G^{2n}$ with $0\le n\le m$ is a sum of $F^{i}G^{k}$ where $i\le 4$ and $i+k$ is both odd and $\le 2m+5$; see the definition of the $u_{i}$.  By Lemma \ref{lemma2.9} the image of any of these elements under $(U+I)^{2}$ is a sum of $F^{i}T(F^{k})$ with $i+k$ odd and $\le 2m+5$. By Lemma \ref{lemma2.5}(b), each such $F^{i}T(F^{k})$ is in the space spanned by the $F^{n}$ with $n$ odd and $\le 2m+1$. It follows that the image of $L$ under $(U+I)^{2}$ has dimension $\le m+1$, and that the dimension of the kernel is $\ge (5m+5)-(m+1)=4m+4$. Since $L\subset L^{*}$, $(U+I)^{2}:L^{*}\rightarrow L^{*}$ has a kernel whose dimension is $\ge 4m+4$, and the dimension of $K_{m}$ is $\ge 2m+2$.
\qed
\end{proof}

\begin{theorem}
\label{theorem2.11}
If $n=6m$ or $6m+2$, $C_{n}$ is a $Z/2$-linear combination of $C_{k}$, $k<n$.
\end{theorem}

\begin{proof}
Let $L$ and $L^{*}$ be as in Definition \ref{def2.7}, and $K_{m}$ be as in Theorem \ref{theorem2.10}. Suppose $f=(r^{2}+r)g(r^{2})$ is a non-zero element of $K_{m}$. Write $g$ as $t^{j}+$ a sum of $t^{k}$ with $k<j$. Applying $U+I$ and using Theorem \ref{theorem1.8}(c) we find that $C_{j}$ is the sum of the corresponding $C_{k}$. So by Lemma \ref{lemma1.10}, $j\equiv 0$ or $2\mod{6}$. Since $j\le 6m+5$, the degree, $2j+2$, of $f$ in $r$ is $2$ or $6\bmod{12}$ and lies in $[0,12m+12]$; this restricts us to $2m+2$ possible degrees. Now $K_{m}$ admits a $Z/2$-basis of elements with distinct degrees in $r$. We've just shown that only the $2m+2$ integers in $\{2,6,14,18,\ldots,12m+2,12m+6\}$ can be degrees. Since the dimension of $K_{m}$ is $\ge 2m+2$, each of these degrees does occur, and in particular there's an $f=(r^{2}+r)g(r^{2})$ in $K_{m}$ with the degree, $n$, of $g$ equal to $6m$ (and also such an $f$ with the degree of $g$ equal to $6m+2$).  Write $g$ as $t^{n}+$ (a sum of $t^{k}$ with $k<n$).  Applying $U+I$ and using Theorem \ref{theorem1.8}(c) we find that $C_{n}$ is the sum of the corresponding $C_{k}$.
\qed
\end{proof}

\begin{corollary}
\label{corollary2.12}
Let $\varphi : Z/2[t]\rightarrow Z/2[t]$ be the $Z/2$-linear map taking $t^{k}$ to $C_{k}$. Then the kernel of $\varphi$ has a basis consisting of $g_{n}$ of degree $n$, one for each $n\equiv 0$ or $2\bmod{6}$.
\end{corollary}

\begin{proof}
Immediate from Lemma \ref{lemma1.10} and Theorem \ref{theorem2.11}.
\qed
\end{proof}

\begin{corollary}
\label{corollary2.13}
Let $K$ be the kernel of $U+I:\modd\rightarrow\modd$. Then the $(r^{2}+r)g_{n}(r^{2})$, $g_{n}$ as in Corollary \ref{corollary2.12}, are a $Z/2$-basis of $K$.
\end{corollary}

\begin{proof}
Theorem \ref{theorem1.8}(c) shows that $(r^{2}+r)g(r^{2})$ is in $K$ if and only if $\varphi(g)= 0$.\\
\qed
\end{proof}

\section{More about the \bm{$g_{n}$}}
\label{section3}

In Corollary \ref{corollary2.12} we introduced certain $g_{n}$ in $Z/2[t]$, $n\equiv 0$ or $2\mod{6}$. We now use the recursion for the $C_{n}$ to show that the $g_{n}$ can be chosen to satisfy certain further conditions. If $g$ is in  $Z/2[t]$, $g=O(t^{m})$ will be shorthand for ``the degree of $g$ is $\le m$''.

\begin{lemma}
\label{lemma3.1}\hspace{2em}\\
\vspace{-2ex}
\begin{enumerate}
\item[(a)] If $n\ge 24$, $C_{n}=t^{24}C_{n-24}+O(t^{n-5})$. When $n$ is even, $n-5$ can be replaced by $n-6$.
\item[(b)] If $n\ge 48$, $C_{n}=t^{48}C_{n-48}+O(t^{n-9})$.
\end{enumerate}
\end{lemma}

\begin{proof}
The recursion of Theorem \ref{theorem1.8} shows that $C_{n}=C_{n-4}+(t^{24}+t^{20}+t^{8}+t^{4})C_{n-24}+t^{n-24}(t^{8}+t^{4})$. Since $C_{m}$ is $O(t^{m-1})$ and is $O(t^{m-2})$ for even $m$, each term on the right other than $t^{24}C_{n-24}$ is $O(t^{n-5})$, and indeed is $O(t^{n-6})$ for even $n$. Similarly we use the fact that $C_{n}=C_{n-8}+(t^{48}+t^{40}+t^{16}+t^{8})C_{n-48}+t^{n-48}(t^{16}+t^{8})$ to get (b).
\qed
\end{proof}

\begin{lemma}
\label{lemma3.2}\hspace{2em}\\
\vspace{-2ex}
\begin{enumerate}
\item[(a)] When $n\equiv 0\mod{12}$, $C_{n}=O(t^{n-4})$.
\item[(b)] When $n\equiv 2\mod{12}$, $C_{n}+C_{n-1}=O(t^{n-4})$.
\item[(c)] When $n\equiv 8\mod{24}$, $C_{n}=O(t^{n-6})$.
\item[(d)] When $n\equiv 20\mod{24}$, $C_{n}+C_{n-2}=O(t^{n-6})$.
\end{enumerate}
\end{lemma}

\begin{proof}
$C_{0}$ and $C_{12}$ are easily seen to be $0$ and $t^{8}+t^{4}+t^{2}$. So (a) holds for $n=0$ and $12$. Suppose $n\ge 24$ and $0\bmod{12}$. Then $C_{n}=t^{24}C_{n-24}+O(t^{n-5})$; by induction this is $O(t^{n-4})$. Also, $C_{2}+C_{1}=0$, while $C_{14}+C_{13}$ has degree $10$. So (b) holds when $n=2$ or $14$. If $n\ge 26$ is $2\bmod{12}$, then $C_{n}+C_{n-1}=t^{24}(C_{n-24}+C_{n-25})+O(t^{n-5})$, which by induction is $O(t^{n-4})$, and we get (b).  Similarly, using the fact that $C_{8}=t^{2}$ and that $C_{20}+C_{18}$ has degree $8$, we use the second sentence in Lemma \ref{lemma3.1}(a) to establish (c) and (d) by induction.
\qed
\end{proof}

\begin{lemma}
\label{lemma3.3}\hspace{2em}\\
\vspace{-2ex}
\begin{enumerate}
\item[(a)] 
When $n\equiv 6\mod{24}$, $C_{n}+C_{n-1}+C_{n-2}+C_{n-3}+C_{n-4}+C_{n-5}=O(t^{n-8})$.\\
When $n\equiv 6\mod{24}$, $C_{n}+C_{n-1}+C_{n-2}+C_{n-3}=O(t^{n-8})$.
\item[(b)] 
When $n\equiv 18\mod{24}$, $C_{n}+C_{n-1}=O(t^{n-8})$.\\
When $n\equiv 18\mod{24}$, $C_{n}+C_{n-1}+C_{n-4}+C_{n-5}=O(t^{n-8})$.
\end{enumerate}
\end{lemma}

\begin{proof}
Lemma \ref{lemma3.2}(b) shows that $C_{n-4}+C_{n-5}=O(t^{n-8})$, so it's enough to check the first parts of (a) and (b).  These first parts are verified using Lemma \ref{lemma3.1}(b), an induction, and the following observations: $C_{6}+C_{5}+C_{4}+C_{3}+C_{2}+C_{1}=0$, $C_{30}+C_{29}+C_{28}+C_{27}+C_{26}+C_{25}$ has degree $22$, $C_{18}+C_{17}$ has degree $8$ and $C_{42}+C_{41}$ has degree $34$.
\qed
\end{proof}

\begin{lemma}
\label{lemma3.4}
The $g_{n}$ of Corollary \ref{corollary2.12} can be chosen so that:
\vspace{-2ex}
\begin{enumerate}
\item[(a)] When $n\equiv 0\mod{12}$, $g_{n}=t^{n}+O(t^{n-2})$.
\item[(b)] When $n\equiv 2\mod{12}$, $g_{n}=t^{n}+t^{n-1}+O(t^{n-2})$.
\item[(c)] When $n\equiv 8\mod{24}$, $g_{n}=t^{n}+O(t^{n-4})$.
\item[(d)] When $n\equiv 20\mod{24}$, $g_{n}=t^{n}+t^{n-2}+O(t^{n-4})$.
\item[(e)] When $n\equiv 6\mod{48}$, $g_{n}=t^{n}+t^{n-1}+t^{n-2}+t^{n-3}+t^{n-4}+t^{n-5}+O(t^{n-6})$.
\item[(f)] When $n\equiv 18\mod{48}$, $g_{n}=t^{n}+t^{n-1}+O(t^{n-6})$.
\item[(g)] When $n\equiv 30\mod{48}$, $g_{n}=t^{n}+t^{n-1}+t^{n-2}+t^{n-3}+O(t^{n-6})$.
\item[(h)] When $n\equiv 42\mod{48}$, $g_{n}=t^{n}+t^{n-1}+t^{n-4}+t^{n-5}+O(t^{n-6})$.
\end{enumerate}
\end{lemma}

\begin{proof}
I'll treat (g)---the other parts are handled similarly. The $\varphi$ of Corollary \ref{corollary2.12} taking $t^{n}$ to $C_{n}$ gives a map:
\footnotesize
\[
\bar{\varphi}:\frac{\mbox{polynomials of degree $\le n$}}{\mbox{polynomials of degree $\le n-6$}}\rightarrow \frac{\mbox{polynomials of degree $\le n-2$}}{\mbox{polynomials of degree $\le n-8$}} 
\]
\normalsize

By Lemma \ref{lemma3.3}, $\bar{\varphi}$ annihilates $t^{n}+t^{n-1}+t^{n-2}+t^{n-3}$. So it suffices to show that every element in the kernel of $\bar{\varphi}$ is represented by some $g_{n}$ in the kernel of $\varphi$. Since $n$ and $n-4$ are $0$ and $2$ mod $6$, there are $g$ of degree $n$ and $n-4$ in the kernel of $\varphi$, and if we can show that the kernel of $\bar{\varphi}$ is spanned by the images of these 2 elements we'll be done. So it's enough to show that the image of $\bar{\varphi}$ has dimension $\ge 4$. But Lemma \ref{lemma1.9} tells us that $\varphi$ maps $t^{n-1},t^{n-2},t^{n-3},t^{n-5}$ to polynomials of degrees $n-2,n-4,n-5,n-6$, completing the proof.
\qed
\end{proof}

\begin{theorem}
\label{theorem3.5}
Let $A$, $B$, $C$ and $D$ be $1$, $t^{6}+t^{5}+t^{4}+t^{3}+t^{2}+t$, $t^{2}+t$ and $t^{8}$.  Then if the $g_{n}$, $n\equiv 0$ or $2\mod{6}$, are chosen as in Lemma \ref{lemma3.4}:
\vspace{-2ex}
\begin{enumerate}
\item[(a)] $g_{12m}=(t^{6}+t^{5})^{2m}\cdot A+O(t^{12m-2})$.
\item[(b)] $g_{12m+6}=(t^{6}+t^{5})^{2m}\cdot B+O(t^{12m})$.
\item[(c)] $g_{12m+2}=(t^{6}+t^{5})^{2m}\cdot C+O(t^{12m})$.
\item[(d)] $g_{12m+8}=(t^{6}+t^{5})^{2m}\cdot D+O(t^{12m+4})$.
\end{enumerate}
\end{theorem}

\begin{proof}
$(t^{6}+t^{5})^{2m}=t^{12m}+O(t^{12m-2})$, while $(t^{6}+t^{5})^{2m}(t^{2}+t)=t^{12m+2}+t^{12m+1}+O(t^{12m})$. So (a) and (b) of Lemma \ref{lemma3.4} give us (a) and (c).  In like manner, (d) for even and odd $m$ follows from (c) and (d) of Lemma \ref{lemma3.4}, while (b) for $m\equiv 0,1,2$ and $3\bmod{4}$ follows from (e), (f), (g) and (h) of that lemma. Suppose for example that $n=12m+6$ and $m=4k+3$, so that $n=48k+42$. Then $(t^{6}+t^{5})^{2m}B=(t^{6}+t^{5})^{8k}(t^{6}+t^{5})^{6}B=t^{48k}(t^{42}+t^{41}+t^{38}+t^{37})+O(t^{48k+36})=t^{n}+t^{n-1}+t^{n-4}+t^{n-5}+O(t^{n-6})$. By Lemma \ref{lemma3.4}(h) this is $g_{n}+O(t^{n-6})=g_{12m+6}+O(t^{12m})$.
\qed
\end{proof}

\section{\bm{$N2$}, \bm{$N2a$} and \bm{$N2b$}. The structure of \bm{$K$}}
\label{section4}

Recall that $\modd\subset Z/2[r]$ consists of the $(r^{2}+r)g(r^{2})$, $g$ in $Z/2[t]$. $\modd$ is stable under multiplication by $r^{2}$, and in particular is a $Z/2[G^{2}]$-module.

\begin{definition}
\label{def4.1}
$N2$ is the (free rank 5) $Z/2[G^{2}]$-submodule of $\modd$ generated by the $u_{0},u_{1},u_{2},u_{4}$ and $u_{5}$ of Lemma \ref{lemma2.6}.
\end{definition}

\begin{lemma}
\label{lemma4.2}
Fix $m\ge 0$ and let $L$ and $L^{*}$ be as in Definition \ref{def2.7}. The kernels of $(U+I)^{2}$ acting on $L^{*}$ and its subspace $L$ are the same.
\end{lemma}

\begin{proof}
The proof of Theorem \ref{theorem2.10} shows that the dimension of the second kernel is $\ge 4m+4$. But the kernel of $U+I:L^{*}\rightarrow L^{*}$ is just the $K_{m}$ of Theorem \ref{theorem2.10}, and the proof of Theorem \ref{theorem2.11} produces a basis of $K_{m}$ with $2m+2$ elements. So the dimension of the first kernel is $\le 2(\,\mbox{the dimension of }K_{m})=4m+4$.
\qed
\end{proof}

\begin{theorem}
\label{theorem4.3}
The kernel, $K$, of $U+I:\modd\rightarrow\modd$ is a subspace of $N2$. So if $g_{n}$, $n\equiv 0$ or $2\bmod{6}$, are as in Corollary \ref{corollary2.12} and $f_{n}=(r^{2}+r)g_{n}(r^{2})$, then each $f_{n}$ is in $N2$.
\end{theorem}

\begin{proof}
If $f$ is in $K$, $f$ is in $L^{*}$ for some $m\ge 0$. Since $(U+I)f=0$, Lemma \ref{lemma4.2} tells us that $f$ is in $L$. But by definition $N2$ contains $L$ for every $m$.
\qed
\end{proof}

\begin{definition}
\label{def4.4}
$N1$ is the $Z/2[G^{2}]$-submodule of $N2$ generated by  $u_{2}=G$. Note that $G^{k}$, $k>0$ and odd, are a basis of $N1$.
\end{definition}

\begin{definition}
\label{def4.5}
$J_{1},J_{7},J_{3}$ and $J_{9}$ are $F,F^{2}G,F^{4}G$ and $F^{8}/G$.
\end{definition}

Note that mod $N1$, $J_{3}=(F+G)^{8}/G=F(F+G)^{2}$. So $J_{3}$, like $J_{1}$, $J_{7}$ and $J_{9}$ lies in $\modd$.

\begin{lemma}
\label{lemma4.6}
$G,J_{1},J_{3},J_{7}$ and $J_{9}$ generate the same $Z/2[G^{2}]$-submodule of $\modd$ as do $G,F,F^{2}G,F^{3}$ and $F^{4}G$. This submodule is in fact $N2$. It follows that $J_{1},J_{3},J_{7}$ and $J_{9}$ are a $Z/2[G^{2}]$ basis of $N2/N1$.
\end{lemma}

\begin{proof}
$J_{1}=F, J_{7}=F^{2}G, J_{9}=F^{4}G$ and $J_{3}=F^{3}+FG^{2}\bmod{N1}$, giving the first result. Also, mod $N1$, $u_{0},u_{4},u_{1}+u_{4}$ and $u_{5}+u_{4}$ are $F,F^{2}G, F^{3}+FG^{2}$ and $F^{4}G+FG^{2}$. So the second result follows.
\qed
\end{proof}

We have defined $J_{k}$ for $k$ in $\{1,3,7,9\}$. We extend the definition to all $k>0$ and prime to $10$ by taking $J_{k+10}$ to be $G^{2}J_{k}$. Lemma \ref{lemma4.6} then shows that the $J_{k}$ form a  $Z/2$-basis of $N2/N1$.

We will use results from \cite{4} to obtain level 5 analogs of the level 3 theorems of \cite{3}. To that end we need to compare the subspace $K$ of $N2$ appearing in Theorem \ref{theorem4.3} with a certain space $N2a$ appearing in a direct sum decomposition of $N2/N1$.

\begin{definition}
\label{def4.7}
$\chi$ is the mod~$20$ Dirichlet character taking $1,3,7,9$ to $1$ and $11,13,17,19$ to $-1$. $N2a$ is spanned by the $J_{k}$ in $N2/N1$ with $\chi(k)=1$. $N2b$ is spanned by the $J_{k}$ in $N2/N1$ with $\chi(k)=-1$.
\end{definition}

Note that $N2a$ is a $Z/2[G^{4}]$-submodule of $N2/N1$ with basis $\{J_{1},J_{3},J_{7},J_{9}\}$, that $N2b$ is a $Z/2[G^{4}]$-submodule with basis $\{J_{11},J_{13},J_{17},J_{19}\}$ and that $N2/N1=N2a\oplus N2b$. Composing the inclusion $K\subset N2$ of Theorem \ref{theorem4.3} with the obvious projection $N2\rightarrow N2a$ coming from the direct sum decomposition, we get a map $K\rightarrow N2a$. We shall show that this map is bijective---a level 5 analog to the level 3 result proved in the last paragraph of section 3 of \cite{3}.

\begin{definition}
\label{def4.8}
An element of $N2/N1$ is $O^{*}(J_{k})$ if it is a sum of $J_{i}$ with $i\le k$.
\end{definition}

The proof of Lemma \ref{lemma4.6} shows that mod $N1$, $u_{0},u_{1},u_{4}$ and $u_{5}$ are $J_{1},J_{7}+J_{3},J_{7}$ and $J_{11}+J_{9}+J_{7}$. So the images of $G^{n}u_{i}$ in $N2/N1$ are $O^{*}(J_{10n+1})$, $O^{*}(J_{10n+7})$, $O^{*}(J_{10n+7})$ and $O^{*}(J_{10n+11})$ according as $i$ is $0,1,4$ or $5$.

\begin{lemma}
\label{lemma4.9}
Suppose $h$ is in $Z/2[t]$ and $f=(r^{2}+r)h(r^{2})$ is in $N2$.
\vspace{-2ex}
\begin{enumerate}
\item[(a)] If $\mathrm{degree}\ h\le 6m+4$, the image of $f$ in $N2/N1$ is $O^{*}(J_{10m+7})$.
\item[(b)] If $\mathrm{degree}\ h\le 12m$, the image of $f$ in $N2/N1$ is $O^{*}(J_{20m+1})$.
\end{enumerate}
\end{lemma}

\begin{proof}
Write $f$ as a sum of $G^{n}u_{i}$, $i$ in $\{0,1,2,4,5\}$.  Then $h$ is the sum of the corresponding $(t^{6}+t^{5})^{n},(t^{6}+t^{5})^{n}\cdot t,(t^{6}+t^{5})^{n}\cdot t^{2},(t^{6}+t^{5})^{n}(t^{4}+t^{3})$ and $(t^{6}+t^{5})^{n}(t^{5}+t^{3})$.  The degrees of these elements are distinct. So in the situation of (a), each of the elements has degree $\le 6m+4$, and each $n$ that appears is $\le m$, with strict inequality when $i=5$. The paragraph following Definition \ref{def4.8} then gives the result. Similarly, in the situation of (b), each element in the sum for $h$ has $n\le 2m$ with inequality when $i=1,2,4$ or $5$, and again we use the paragraph following Definition \ref{def4.8}.
\qed
\end{proof}

\begin{theorem}
\label{theorem4.10}
Let the $g_{n}$ of Corollary \ref{corollary2.12}, $n\equiv 0$ or $2\mod{6}$, be chosen as in Theorem \ref{theorem3.5}, and let $f_{n}=(r^{2}+r)g_{n}(r^{2})$. Recall that the $f_{n}$ are a $Z/2$-basis of $K$. Consider the composite map $K\subset N2\rightarrow N2a$ described after Definition \ref{def4.7}. Then:
\vspace{-2ex}
\begin{enumerate}
\item[(a)] $f_{12m}$ maps to $J_{20m+1}+O^{*}(J_{20m-11})$.
\item[(b)] $f_{12m+6}$ maps to $J_{20m+3}+O^{*}(J_{20m+1})$.
\item[(c)] $f_{12m+2}$ maps to $J_{20m+7}+O^{*}(J_{20m+3})$.
\item[(d)] $f_{12m+8}$ maps to $J_{20m+9}+O^{*}(J_{20m+7})$.
\item[(e)] The map $K\subset N2\rightarrow N2a$ is bijective.
\end{enumerate}
\end{theorem}

\begin{proof}
It's enough to prove (a), (b), (c), (d). For if they hold, our map takes a basis of $K$ to a basis of $N2a$.
\vspace{-2ex}
\begin{enumerate}
\item[(a)] $g_{0}=1$ and so $f_{0}=r^{2}+r=F+G$ with image $J_{1}$. Suppose $m>0$. By Theorem \ref{theorem3.5}, $g_{12m}=(t^{6}+t^{5})^{2m}+g^{*}$ with $\mathrm{degree}\ g^{*}\le 6(2m-1)+4$. Let $f^{*}=(r^{2}+r)g^{*}(r^{2})$. Then $f_{12m}=G^{4m}u_{0}+f^{*}$. Since $f_{12m}$ is in $N2$, so is $f^{*}$.  The image of $u_{0}$ is $J_{1}$. So the image of $G^{4m}u_{0}$ is $J_{20m+1}$, and Lemma \ref{lemma4.9}(a) shows that the image of $f^{*}$ in $N2/N1$ is $O^{*}(J_{20m-3})$. Since $J_{20m-3},J_{20m-7}$ and $J_{20m-9}$ are all in $N2b$, the image of $f^{*}$ in $N2a$ is in fact $O^{*}(J_{20m-11})$.
\item[(c)] By Theorem \ref{theorem3.5}, $g_{12m+2}=(t^{6}+t^{5})^{2m}(t^{2}+t)+g^{*}$ with $\mathrm{degree}\ g^{*}\le 12m$. Let $f^{*}=(r^{2}+r)g^{*}(r^{2})$. Then $f_{12m+2}=G^{4m}(u_{2}+u_{1})+f^{*}$, and once again $f^{*}$ is in $N2$. The paragraph following Definition \ref{def4.8} shows that the image of $u_{1}$ is $J_{7}+J_{3}$. So the image of $G^{4m}u_{1}$ is $J_{20m+7}+J_{20m+3}$. Also, Lemma \ref{lemma4.9}(b) shows that the image of $f^{*}$ is $O^{*}(J_{20m+1})$.
\item[(b)] By Theorem \ref{theorem3.5}, $g_{12m+6}=(t^{6}+t^{5})^{2m}(t^{6}+t^{5}+t^{4}+t^{3}+t^{2}+t)+g^{*}$ with $\mathrm{degree}\ g^{*}\le 12m$. Let $f^{*}=(r^{2}+r)g^{*}(r^{2})$. Then $f_{12m+6}=G^{4m}(G^{2}u_{0}+u_{4}+u_{2}+u_{1})+f^{*}$, and once again $f^{*}$ is in $N2$. The image of $G^{4m+2} u_{0}$ in $N2/N1$ is $J_{20m+11}$ which lies in $N2b$. The paragraph following Definition \ref{def4.8} shows that the image of $u_{4}+u_{1}$ is $J_{3}$. So the image of $G^{4m}(u_{4}+u_{1})$ is $J_{20m+3}$. And as in (c), the image of $f^{*}$ is $O^{*}(J_{20m+1})$
\item[(d)] By Theorem \ref{theorem3.5}, $g_{12m+8}=(t^{6}+t^{5})^{2m}(t^{8}+t^{3})+g^{*}$ with $\mathrm{degree}\ g^{*}\le 12m+4$. Let $f^{*}=(r^{2}+r)g^{*}(r^{2})$. Observing that $t^{8}+t^{3}=(t^{6}+t^{5})(t^{2}+t+1)+(t^{5}+t^{3})$, we find that $f_{12m+8}=G^{4m}(G^{2}(u_{2}+u_{1}+u_{0})+u_{5})+f^{*}$, so that $f^{*}$ is in $N2$. The image of $u_{1}+u_{0}$ in $N2/N1$ is $J_{7}+J_{3}+J_{1}$, and the image of $G^{4m+2}(u_{2}+u_{1}+u_{0})$ therefore lies in $N2b$. Since the image of $u_{5}$ in $N2/N1$ is $J_{11}+J_{9}+J_{7}$, the image of $G^{4m}u_{5}$ is $J_{20m+11}+J_{20m+9}+J_{20m+7}$ which projects to $J_{20m+9}+J_{20m+7}$ in $N2a$.  Lemma \ref{lemma4.9}(a) shows that the image of $f^{*}$ is $O^{*}(J_{20m+7})$ completing the proof.
\end{enumerate}
\qed
\end{proof}

\section{The action of \bm{$T_{p}$} and the main theorem}
\label{section5}

Following the ideas sketched in the introduction we use Theorems \ref{theorem4.3} and \ref{theorem4.10}(e) to derive a result about a certain Hecke algebra in level $\Gamma_{0}(5)$. We refer extensively to \cite{4}. Instead of $r$ being an indeterminate it is now the explicit element $\sum_{n>0}(x^{n^{2}}+x^{2n^{2}}+x^{5n^{2}}+x^{10n^{2}})$ of $Z/2[[x]]$. Then $Z/2[r]$ is a subspace of $Z/2[[x]]$, and Theorem 1.11 of \cite{4} shows that it is the space, $M$, of mod~2 modular forms of level  $\Gamma_{0}(5)$. That theorem also shows that the subspace $\modd$ of $M$ consisting of odd power series lying in $M$ is just the $\modd$ of the present paper, spanned by the $(r^{2}+r)r^{2n}$.

Now for $p\ne 2$ or $5$ we have formal Hecke operators $T_{p}:Z/2[[x]]\rightarrow Z/2[[x]]$. They commute and stabilize $M$ and $\modd$; see Theorem 1.14 of \cite{4} and the paragraph preceding it.

\begin{definition}
\label{def5.1}
$U_{5}:Z/2[[x]]\rightarrow Z/2[[x]]$ takes $\sum c_{n}x^{n}$ to $\sum c_{5n}x^{n}$. 
\end{definition}

\begin{lemma}
\label{lemma5.2}
$U_{5}$ commutes with the $T_{p}$ and stabilizes $M$ and $\modd$.
\end{lemma}

\begin{proof}
It's enough to show that $U_{5}$ stabilizes the space spanned by the $r^{i}$ with $0\le i\le 2m$. We argue as in the proof of Theorem 1.14 of \cite{4}, using the classical Hecke operator $\sum c_{n}x^{n}\rightarrow \sum c_{5n}x^{n}$ on weight $4m$ holomorphic modular forms of level $\Gamma_{0}(5)$.
\qed
\end{proof}

\begin{lemma}
\label{lemma5.3}
$F=\sum_{n\ \mathit{odd},\ n>0} x^{n^{2}}$, and $G=\sum_{n\ \mathit{odd},\ n>0} x^{5n^{2}}$.
\end{lemma}

\begin{proof}
See Theorem 1.12 of \cite{4} and the paragraph preceding it.
\qed
\end{proof}

\begin{lemma}
\label{lemma5.4}
$U_{5}:M\rightarrow M$ is the map $U$ of Definition \ref{def1.3}.
\end{lemma}

\begin{proof}
Lemma \ref{lemma5.3} and the definition of $U_{5}$ show that $U_{5}(Gf)=FU_{5}(f)$. So $U_{5}:M\rightarrow M$ like $U$ is semi-linear, and it's enough to show that they agree on the basis $1,r,r^{2},r^{3},r^{4},r^{5}$ of $Z/2[r]$ as $Z/2[G]$-module. One sees directly from the definitions that $U_{5}(r)=r$. Then $U_{5}$ and $U$ each fix $1,r,r^{2},r^{4}$ and $r^{8}$. Since $r^{5}=r^{8}+G((r^{2}+r+1)$, we may use semi-linearity to see that $U_{5}(r^{5})=U(r^{5})$.  Then $U_{5}(r^{10})=U(r^{10})$, and since $G(r^{4}+r^{3})=r^{10}+r^{8}$ another application of semi-linearity shows that $F\cdot U_{5}(r^{3})=F\cdot U(r^{3})$.

Recall now that $N2$ and $N1$ are the $Z/2[G^{2}]$-submodules of $\modd$ with bases $\{G,F,F^{2}G,F^{3},F^{4}G\}$ and $\{G\}$, and that $N2a$ and $N2b$ are the subspaces of $N2/N1$ spanned by the $J_{k}$ in $\modd$ with $\chi(k)=1$ and $-1$ respectively. $N2$ and $N1$ also appear in Definition 1.15 of \cite{4} and coincide with the $N2$ and $N1$ just described. Also, the $J_{k}$ with $(k,10)=1$ appear in Definitions 1.16 and Theorem 1.17 of \cite{4}. Those definitions, in terms of $F$ and $G$, show that the $J_{k}$ are just the $J_{k}$ in our section \ref{section4}. It follows that the subspaces $N2a$ and $N2b$ of $N2/N1$ appearing in the paragraph following the proof of Lemma 2.13 of \cite{4} are the $N2a$ and $N2b$ of our section \ref{section4}.
\qed
\end{proof}

\begin{lemma}
\label{lemma5.5}
Let $K$ be the kernel of $U_{5}+I$ acting on the space $\modd$ of odd mod~2 modular forms of level $\Gamma_{0}(5)$.  Let $N2,N1,N2a$ and $N2b$ be as in \cite{4}. Then $K\subset N2$ and the composition of $N2\rightarrow N2/N1$ with the projection map $N2/N1=N2a\oplus N2b\rightarrow N2a$ maps $K$ bijectively to $N2a$.
\end{lemma}

\begin{proof}
We have seen that $\modd$ is just the subspace of $Z/2[r]$ spanned by the $(r^{2}+r)r^{2n}$. Furthermore, by Lemma \ref{lemma5.4}, $K$ is the kernel of $U+I$ acting on this space. The results now follow from Theorem \ref{theorem4.3}, Theorem \ref{theorem4.10}(e), and the identification of the $N2,N1,N2a$ and $N2b$ of \cite{4} with the $N2,N1,N2a$ and $N2b$ defined in this paper.
\qed
\end{proof}

Now Lemma \ref{lemma5.2} shows that the $T_{p}:Z/2[[x]]\rightarrow Z/2[[x]]$, $p\ne 2$ or $5$, stabilize not only $\modd$, but also the $K$ of Lemma \ref{lemma5.5}. Also by Theorem 2.19 of \cite{4} and the remark following Theorem 1.17 of \cite{4}, these $T_{p}$ also stabilize $N2$ and $N1$, and therefore act on $N2/N1=N2a\oplus N2b$. Since $\chi(3)=\chi(7)=1$, Corollary 2.18 of \cite{4} shows that $T_{3}$ and $T_{7}$ stabilize $N2a$ and $N2b$.

\begin{lemma}
\label{lemma5.6}
The bijection $K\rightarrow N2a$ of Lemma \ref{lemma5.5} preserves the action of $T_{3}$ and $T_{7}$.
\end{lemma}

\begin{proof}
It suffices to show that the projection map $N2a\oplus N2b\rightarrow N2a$ preserves the action of $T_{3}$ and $T_{7}$. But $T_{3}$ and $T_{7}$ stabilize both $N2a$ and $N2b$.
\qed
\end{proof}

Now let $\mathit{pr}:N2\rightarrow Z/2[[x]]$ be the map $\sum c_{n}x^{n}\rightarrow \sum_{(n,5)=1}c_{n}x^{n}$. This map annihilates $N1$ and gives a map $\mathit{pr}:N2/N1 \rightarrow Z/2[[x]]$. Let $D=\mathit{pr}(J_{1})=\sum_{(n,10)=1,\, n>0}x^{n^{2}}$.

In section 2 of \cite{4} we showed that $\mathit{pr}$ maps $N2a$ bijectively to a $Z/2[G^{4}]$-submodule of $Z/2[[x]]$, denoted by $W_{a}$, generated by $D,D^{8}/G,D^{2}G$ and $D^{4}G$.  Furthermore $T_{3}$ and $T_{7}$ stabilize $W_{a}$, and $\mathit{pr}:N2a\rightarrow W_{a}$ evidently preserves the action of $T_{3}$ and $T_{7}$. We conclude that the composite map $K\rightarrow N2a\rightarrow W_{a}$ is a bijection preserving the action of $T_{3}$ and $T_{7}$. Note also that this bijection maps the element $r^{2}+r=F+G$ of $K$ to $D$.

We will now use a deep result from \cite{4} to prove the following main theorem.

\begin{theorem}
\label{theorem5.7}
Let $K$ be as in Lemma \ref{lemma5.5}. Then there are $m_{i,j}$ in $K$ such that:
\vspace{-2ex}
\begin{enumerate}
\item[(a)] $m_{0,0}=r^{2}+r=F+G$.
\item[(b)] $T_{3}(m_{i,j})=m_{i-1,j}$ or $0$ according as $i>0$ or $i=0$.
\item[(c)] $T_{7}(m_{i,j})=m_{i,j-1}$ or $0$ according as $j>0$ or $j=0$.
\item[(d)] The $m_{i,j}$ are a $Z/2$-basis of $K$.
\end{enumerate}
\end{theorem}

\begin{proof}
In view of the bijection of $K$ with $W_{a}$, preserving the action of $T_{3}$ and $T_{7}$, it suffices to prove the above result with $K$ and $F+G$ replaced by $W_{a}$ and $D$. The result for $W_{a}$ is precisely Corollary 5.13 of \cite{4}. 
\qed
\end{proof}

Now there is an action of $Z/2[X,Y]$ on $K$ with $X$ and $Y$ acting by $T_{3}$ and $T_{7}$. Theorem \ref{theorem5.7} shows that $(X,Y)$ acts ``locally nilpotently'', i.e., each $f$ in $K$ is annihilated by some power of $(X,Y)$. So we get an action of $Z/2[[X,Y]]$, with $X$ and $Y$ acting by $T_{3}$ and $T_{7}$.

\begin{theorem}
\label{theorem5.8}
The above action is faithful. Furthermore if $p\ne 2$ or $5$ then $T_{p}:K\rightarrow K$ is multiplication by some element of the maximal ideal $(X,Y)$ of $Z/2[[X,Y]]$.
\end{theorem}

\begin{proof}
The first result is an easy consequence of Theorem \ref{theorem5.7}. If $p\ne 2$ or $5$, $T_{p}$ commutes with $T_{3}$ and $T_{7}$ and so is $Z/2[[X,Y]]$-linear. Theorem \ref{theorem5.7} then shows (see the proof of Theorem 4.16 of \cite{2}) that $T_{p}:K\rightarrow K$ is multiplication by some $u$ in $Z/2[[X,Y]]$. Since $T_{p}(m_{0,0})=T_{p}(F+G)=0$, $u$ is in $(X,Y)$.
\qed
\end{proof}

Finally, since the $T_{p}$, $p\ne 2$ or $5$, act on $K$, we have a ``completed shallow Hecke algebra'', $\mathit{HE}(K)$, attached to $K$. Theorem \ref{theorem5.8} tells us that $\mathit{HE}(K)$ is a power series ring in $T_{3}$ and $T_{7}$. At the same time the action of these $T_{p}$ on $N2/N1$ gives rise to a completed shallow Hecke algebra $\mathit{HE}(N2/N1)$. Theorem 7.2 of \cite{4} tells us that $\mathit{HE}(N2/N1)$ is a power series ring in $T_{3}$ and $T_{7}$ with an element of square $0$ adjoined. Arguing as in the final paragraph of \cite{3} we find:

\begin{theorem}
\label{theorem5.9}
There is an isomorphism $(\mathit{HE}(N2/N1))_{\mathit{red}}\rightarrow \mathit{HE}(K)$ taking $T_{p}$ to $T_{p}$ for each $p\ne 2$ or $5$. 
\end{theorem}

% \appendix

% \section{}
% \label{}

\end{document}